\documentclass[12pt, oneside]{amsart}
\usepackage[left=3cm,right=2cm,top=3cm,bottom=2cm]{geometry}
\usepackage{graphicx,color}
\usepackage{multirow}
\usepackage{amsmath}
\usepackage[english]{babel}
\usepackage{amssymb}
\usepackage{epsfig}
\usepackage{graphicx,color}%\textcolor{red}{texto em vermelho}
\usepackage[latin1]{inputenc}%acentos
\usepackage{latexsym}%\Box
\usepackage[normalem]{ulem}%sublinhados
\usepackage{graphpap}
\usepackage{bm}

\usepackage[shortlabels]{enumitem}%para mudar enumerate
\usepackage[pdftex]{hyperref}%links
\usepackage[pdftex]{hyperref}

\DeclareMathOperator{\rank}{\mathrm{rank}}

\newtheorem{teo}{Theorem}[section]

\newtheorem{defi}{Definition}[section]
\newtheorem{obs}{Remark}[section]
\newtheorem{lema}{Lemma}[section]
\newtheorem{prop}{Proposition}[section]

\newcommand{\bzero}{\mathbf{0}}
\newcommand{\isEquivTo}[1]{\underset{#1}{\sim}}
\newcommand{\bnu}{\bm{\nu}}

\newcommand{\cP}{\mathcal{P}}
\newcommand{\cA}{\mathcal{A}}
\newcommand{\cB}{\mathcal{B}}

\newcommand{\cO}{\mathcal{O}}

\newcommand{\cH}{\mathcal{H}}

\newcommand{\cF}{\mathcal{F}}
\newcommand{\tF}{\widetilde{F}}
\newcommand{\cT}{\mathcal{T}}
\newcommand{\cK}{\mathcal{K}}
\newcommand{\cG}{\mathcal{G}}

\newcommand{\cW}{\mathcal{W}}

\newcommand{\K}{\mathcal{K}}
\newcommand{\bxi}{\bm{\xi}}
\newcommand{\bd}{\bm{\delta}}

\newcommand{\bx}{\mathbf{x}}
\newcommand{\bX}{\mathbf{X}}
\newcommand{\bY}{\mathbf{Y}}
\newcommand{\bA}{\mathbf{A}}
\newcommand{\bv}{\mathbf{v}}
\newcommand{\bw}{\mathbf{w}}

\newcommand{\cal}{\mathcal}
\newcommand{\R}{\mathbb{R}}

\newcommand{\NN}{\mathbb{N}}

\title{Generic singularities of affine distance functions and plane congruences}

\author{Igor Chagas Santos}

\address{
	(Igor Chagas Santos) Universidade Federal da Bahia, IME, Departamento de Matemática, Av. Milton Santos, S/N Ondina CEP: 40170-110, Salvador-BA, Brazil.}
\email{ igorchagas@ufba.br}

\begin{document}
\begin{abstract}
In this paper, we classify the generic singularities of 2-parameter
plane congruences in $\R^4$ and the generic singularities of affine normal plane congruences. We also study the generic singularities of the family of affine distance functions.
	
\end{abstract}

\maketitle
\section{Introduction}
The study of line congruences from the singularity theory viewpoint has been widely explored in recent years. We refer the reader to \cite{barajas2020lines}, \cite{bruce2023binary}, \cite{craizer2023singularities}, \cite{Giblin}, \cite{Izumiya} and \cite{nossoartigo}, for instance. On the other hand, the study of plane congruences in $\R^4$, that is nothing but a $2$-parameter family of planes in $\R^4$ over a surface $M \subset \R^4$, has not received the same attention. Here, we explore the study of plane congruences in the general case, when transversal planes are considered and also in the affine normal case, when affine normal planes are considered.

Locally, a plane congruence is a smooth map $F_{(\bx, \bxi, \bd)}: U \times I \times J \to \R^4$, given by $F_{(\bx, \bxi, \bd)}(u,t,l) = \bx(u) + t\bxi(u) + l\bd(u)$, where $U \subset \R^2$ is open, $I$, $J$ are open intervals, $\bx: U \to \R^2$ is smooth and $\bxi, \bd: U \to \R^4 \setminus \lbrace 0 \rbrace$ are smooth and linearly independent. We then classify singularities of plane congruences in the general case and in the affine case, when the family of affine normal planes is considered. In this last case, we need to study surfaces in $\R^4$ from the affine differential geometry viewpoint.

The study of surfaces in $\R^4$ from the affine viewpoint is related to the famous Carathéodory conjecture which states that any convex compact surface $\tilde{M} \subset \R^3$ has at least two umbilic points. It happens that this conjecture is related to a Loewner conjecture which states that any locally strictly convex surface $M \subset \R^4$ homeomorphic to the sphere has at least two inflections (see \cite{gutierrez2003indices}). As mentioned in \cite{juanjo1}, a positive answer to this conjecture would imply a positive answer to the Carathéodory conjecture. %Thus, as this conjecture is related to affine invariants and surfaces in $\R^4$, studying surfaces in $\R^4$ from the affine viewpoint and affine plane congruences it is worth it.

The case of surfaces in $\R^4$ from the affine geometry viewpoint was studied, for instance, in \cite{craizer2015}, \cite{craizer2017}, \cite{craizer2018}, \cite{nomizu1993new}, \cite{juanjo2} and \cite{juanjo1}. In \cite{juanjo2}  the authors introduce a new definition of affine normal plane $\bA$, a family of affine distance functions $\Delta$ and characterize singularities of affine distance functions in terms of the affine normal plane. Here, taking this into account, we provide a definition of affine normal plane congruences, using the affine normal plane $\bA$ and explore these objects from the singularity theory viewpoint.

In section \ref{sec3}, working with the family of affine distance functions $\Delta$ we show that this family is a Morse family of functions, in proposition \ref{propmorse}, which connects our study to the study of Lagrangian singularities, since every germ of Lagrangian map arises from a Morse family. With this, we show in proposition \ref{Prop5.4} that $\Delta$ is locally 
$\cP$-$\mathcal{R}^{+}$-versal and in theorem \ref{classifsingdelta} we provide the generic singularities of $\Delta$, using Mather nice dimensions. We also investigate how these singularities affect the geometry of a locally strictly convex surface. For instance, in theorem \ref{caracsemi} we characterize affine semiumbilic points as those which are singularities of corank $2$ of some affine distance function. The ridges of order $4$ and $5$ are also studied in proposition \ref{propridge}.

In section \ref{sec4}, we classify singularities of $2$-parameter plane congruences and $2$-parameter affine normal plane congruences in theorems \ref{teo4.1} and \ref{teoprincipal}, respectively. When working with the affine plane congruences case, we use the theory of Lagrangian singularities, because the family of affine distance functions is a Morse family and its Lagrangian map is exactly the germ of affine normal plane congruence. Furthermore, in theorem \ref{caractsing} we characterize singularities of affine normal plane congruences.

\section{Preliminaries}

\subsection{Basic results on affine differential geometry}
Let us consider $(M, \nabla)$ and $(\tilde{M}, \tilde{\nabla})$ two smooth manifolds of dimensions $m$ and $n$, respectively. Let $k = n-m > 0$.
\begin{defi}\label{defaffineimmersion}\normalfont
	A differentiable immersion $f: M \to \tilde{M}$ is said to be an affine immersion if there is a $k$-dimensional differentiable distribution $\sigma$ along $f$, that is, for each $x \in M$ we associate $\sigma_{x}$ a subspace of $T_{f(x)}(\tilde{M})$ such that
	\begin{align*}
		T_{f(x)}(\tilde{M}) = f_{\ast}(T_{x}(M)) \oplus \sigma_{x}.
	\end{align*}
	and such that for all $X, Y$ smooth vector fields tangent to $M$, we have
	\begin{align}\label{decompo}
		\left(\tilde{\nabla}_{X}f_{\ast}(Y)\right)_{x} = \left(f_{\ast}\left( \nabla_{X}Y
		\right)  \right)_{x} + \left(\alpha(X,Y) \right)_{x},\; \text{where $\alpha(X,Y)_{x} \in \sigma_{x}$, $x \in M$}.
	\end{align}
\end{defi}

Taking into account definition \ref{defaffineimmersion}, if we take, for instance, a surface patch $M$ given by $\bx: U \to \R^4$, $X, Y$ vector fields on $M$ for which we denote the smooth extensions to $\R^4$ in the same way and $\bxi_{1}, \bxi_{2}$ smooth vector fields such that $\sigma_{q} = \text{span}\lbrace \bxi_{1}(q), \bxi_{2}(q) \rbrace$, for all $q \in U$, we have that
\begin{subequations}\label{decompcampos}
	\begin{align}
		D_{X}Y &= \nabla_{X}Y + h_{1}(X,Y)\bxi_{1} + h_{2}(X,Y)\bxi_{2}\\
		D_{X}\bxi_{i} &= -S_{i}X + \tau_{i}^{1}(X)\bxi_{1} + \tau_{i}^{2}(X)\bxi_{2},\; i=1,2.
	\end{align}
\end{subequations}
where $\nabla_{X}Y$ denotes the tangent component of $	D_{X}Y$, $\nabla$ is the torsion free induced connection, $h_{1}, h_2$ are symmetric bilinear forms, $S_{1}, S_2$ are $(1,1)$ tensor fields, called shape operators associated to $\bxi_{1}$ and $\bxi_{2}$, respectively and $\tau_{i}^{j}$ is a $1$-form, $i,j=1,2$.

\subsection{Transversality, Unfoldings, Morse family of functions and Lagrangian singularities}
We now present some basic results in singularity theory which help us in the next sections. More details can be found in \cite{gg},  \cite{Livro} and \cite{juanjo}.

\begin{lema}{\rm{({\cite{gg}, Lemma 4.6})}}\textrm{(Basic Transversality Lemma)}\label{basictransvlemma}
	Let $X$, $B$ and $Y$ be smooth manifolds with $W$ a submanifold of $Y$. Consider $j: B \rightarrow C^{\infty}\left(X, Y \right)$ a non-necessarily continuous map and define $\Phi: X \times B \rightarrow Y$ by $\Phi(x,b) = j(b)(x)$. Suppose $\Phi$ smooth and transversal to $W$, then the set
	\begin{align*}
		\left\lbrace b \in B: j(b) \pitchfork W \right\rbrace
	\end{align*}
	is a dense subset of $B$.
\end{lema}

\begin{defi}\label{defunfolding}\normalfont
	Let $f: (N, x_{0}) \rightarrow  (P, y_{o})$ be a map germ between manifolds. An \textit{unfolding} of $f$ is a triple $(F,i,j)$ of map germs, where $i: (N, x_0) \rightarrow (N', x'_{0})$, $j:(P, y_0) \rightarrow (P', y'_{0}) $ are immersions and $j$ is transverse to $F$, such that $F \circ i = j \circ f$ and $(i,f): N \rightarrow \lbrace (x', y)\in N' \times P: F(x') = j(y) \rbrace $ is a diffeomorphism germ (associated diagram in figure \ref{diagrama}). The dimension of the unfolding is $dim(N') - dim(N)$.
	\begin{figure}[h]
		\begin{center}
			\includegraphics[scale=0.2]{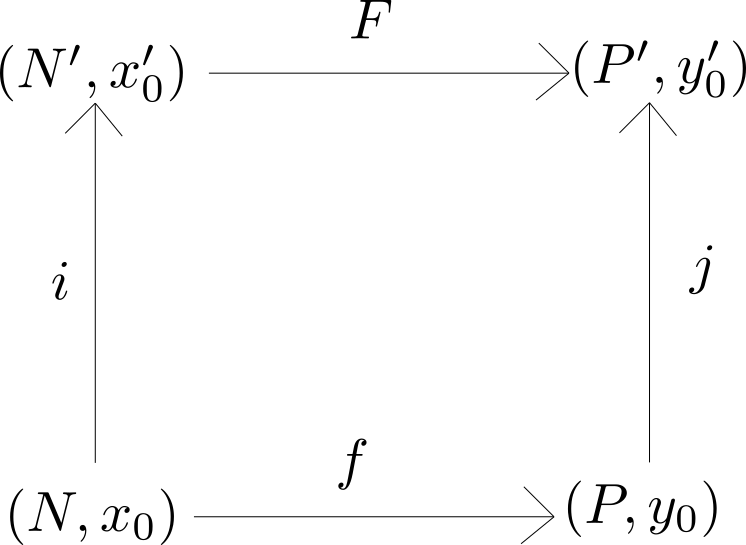}
			\caption{Associated diagram}\label{diagrama}
		\end{center}
	\end{figure}
\end{defi}
\begin{lema}{\rm{(\cite{Izumiya}, Lemma 3.1)}}\label{lema3.1}
	Let $F: (\R^{n-1} \times \R, (\bm{0}, 0)) \rightarrow (\R^n, \bm{0})$ be a map germ with components $F_{i}(x,t)$, $i=1,2,\cdots, n$, i.e.
	\begin{align*}
		F(x,t) = (F_{1}(x,t), \cdots, F_{n}(x,t)).
	\end{align*}
	Suppose that $\dfrac{\partial F_{n}}{\partial t}(0,0) \neq 0$. We know by the Implicit Function Theorem that, there is a germ of function $g: (\R^{n-1}, \bm{0}) \rightarrow (\R, 0)$, such that
	\begin{align*}
		F_{n}^{-1}(0) = \lbrace (x, g(x)): x \in (\R^{n-1}, \bm{0})  \rbrace.
	\end{align*}
	Let us consider the immersion germs $i: (\R^{n-1}, \bm{0}) \rightarrow (\R^n, (\bm{0},0))$, given by $i(x) = (x,g(x))$, $j: (\R^{n-1}, \bm{0}) \rightarrow (\R^n, (\bm{0}, 0))$, given by $j(y) = (y,0)$ and a map germ $f: (\R^{n-1}, \bm{0}) \rightarrow (\R^{n-1}, \bm{0})$, given by $f(x) = (F_{1}(x,g(x)), \cdots, F_{n-1}(x,g(x)))$. Then the triple $(F,i,j)$ is a one-dimensional unfolding of $f$.
\end{lema}

\begin{lema}{\rm{(\cite{izumiyaruled}, Lemma 3.3})}\label{lema3.2}
	Let $F : (\R^n \times \R^r, \bm{0}) \rightarrow (\R^p \times \R^r, \bm{0})$ be an unfolding of $f_{0}$ of the form
	$F (q, x) = (f(q, x), x)$. If $j^{k}_{1}f$ is transverse to $\cK^{k}(j^{k}f_{0}(0))$ for a sufficiently large $k$, then $F$ is infinitesimally $\cA$-stable.
\end{lema}

Let $f: X \to Y$ be a smooth map between manifolds. We say that $f$ has a singularity of type $S_{r}$ at $p$ if the Jacobian matrix of $f$ at $p$ has kernel rank $r$. We denote by $S_{r}(f)$ the subset of $X$ of the singularities of type $S_r$. For a generic $f$, the subsets $S_{r}(f)$ are submanifolds of $X$ of codimension $r^2 + er$, where $e = \left| \dim X - \dim Y \right|$. When $S_{r}(f)$ is a submanifold we can take the restriction of $f$ to $S_{r}(f)$ and we denote by $S_{r,f}$ the subset of $S_{r}(f)$ where the Jacobian matrix of the restriction has kernel rank $s$. Generically, these subsets are submanifolds of $S_{r}(f)$. By continuing this process, we obtain a family of submanifolds $S_{i_{i}, \cdots, i_{k}}(f)$. We say that a point in $S_{i_{i}, \cdots, i_{k}}(f)$ is a singularity of $f$ with \textit{Thom-Boardman symbol} $S_{i_{i}, \cdots, i_{k}}$. Note that $S_{1}(f)$ consists of the singular points of type $A_{k\geq 1}$, while the points in $S_{1,1}(f)$ are the singularities of type $A_{k\geq 2}$, the points in $S_{1,1,1}(f)$ are the singularities of type $A_{k\geq 3}$ and we can continue this process. For details about  Thom-Boardman symbols, see \cite{gg}.

\begin{defi}\normalfont
	We say that a $r$-parameter family of germs of functions $G : (\R^n \times \R^r, \bm{0}) \rightarrow (\R, 0)$ is a \textit{Morse family of functions} if the map germ $\Delta_{G}:(\R^n \times \R^r, \bm{0}) \rightarrow (\R^n, \bm{0}) $, given by
	\begin{align*}
		\Delta_{G}(q,x) = \left( \frac{\partial G}{\partial q_{1}}, \cdots, \frac{\partial G}{\partial q_{n}}\right)(q,x)
	\end{align*}
	is not singular.
\end{defi}

\begin{defi}\normalfont
	The \textit{catastrophe set} of a family of functions $G : (\R^n \times \R^r, \bm{0}) \rightarrow (\R, 0)$ is the set
	\begin{align*}
(C_{G}, 0) = \left\lbrace (q,x) \in (\R^n \times \R^r, \bm{0}): \frac{\partial G}{\partial q_{1}}(q,x) =  \cdots, \frac{\partial G}{\partial q_{n}}(q,x) = 0 \right\rbrace.
	\end{align*}
\end{defi}

If $G$ is a Morse family of functions, we have that $(C_{G}, 0)$ is a germ of smooth manifold of $ (\R^n \times \R^r, \bm{0})$. It is possible to show that the map $L(G): (C_{G}, 0) \to T^{*}\R^r$, given by 
\begin{align*}
	L(G)(q,x) = \left( x,  \frac{\partial G}{\partial x_{1}}, \cdots, \frac{\partial G}{\partial x_{r}}\right)
\end{align*}
 is a Lagrangian immersion (see \cite{Livro} section 5.1.2 for details) and $\pi \circ L(G)$ is the associated Lagrangian map, where $\pi : T^{*}\R^r \to \R^r$. It is also possible to show that every germ of Lagrangian immersion is given by $L(G)$, where $G$ is a Morse family of functions (see theorem 5.3 in \cite{Livro}). 
 
 \begin{defi}\normalfont
 	Let $\cG$ be one of Mather's subgroups of $\cK$ and $\cB$ a smooth manifold. A family of maps $G: \R^n \times \cB \rightarrow \R^k$, given by $G(x,u) = g_{u}(x) $, is said to be \textit{locally $\cG$-versal}  if for every $(x,u) \in \R^n \times \cB$, the germ of $G$ at $(x,u)$ is a $\cG$-versal unfolding of $g_{u}$ at $x$.
 \end{defi}

 \begin{teo}{\rm{(\cite{Livro}, Theorem 5.4})}\label{teo5.4}
 	The Lagrangian map-germ $\pi \circ L(G)$ is Lagrangian stable if and only
 	if $G$ is an $\mathcal{R}^{+}$- versal unfolding of $g(q) = G(q, 0)$.
 \end{teo}

 \begin{teo}{\rm{(\cite{Livro}, Theorem 5.5})}\label{teo5.5}
Let $G: \R^n \times \cB \rightarrow \R^k$ be a Morse family of functions. Suppose that $L(G) : (C(G), 0) \to T^{*}\R^r$ is Lagrangian stable and $n \leq  4$. Then $L(G)$ is Lagrangian equivalent to a germ of a Lagrangian submanifold whose generating family $\tilde{G}(q_1, \cdots, q_n , x_1 , \cdots, x_r)$ is one of the following germs, where $Q(q_s , \cdots, q_n ) = \pm (q_r)^2 \pm \cdots \pm( q_n)^2$,
\begin{enumerate}[(1)]
	\item $Q(q_{2}, \cdots, q_{n}) + q_{1}^{3} + x_{1}q_{1}$ ($A_{2}$-singularity)
	\item $Q(q_{2}, \cdots, q_{n}) + q_{1}^{4} + x_{1}q_{1} + x_{2}q_{1}^2$ ($A_{3}$-singularity)
	\item $Q(q_{2}, \cdots, q_{n}) + q_{1}^{5} + x_{1}q_{1} + x_{2}q_{1}^2 + x_{3}q_{1}^3$ ($A_{4}$-singularity)
	\item $Q(q_{2}, \cdots, q_{n}) + q_{1}^{6} + x_{1}q_{1} + x_{2}q_{1}^2 + x_{3}q_{1}^3 + x_{4}q_{1}^4  $ ($A_{5}$-singularity)
	\item $Q(q_{3}, \cdots, q_{n}) + q_{1}^{3} + q_{1}q_{2}^2 + x_{1}q_{1} + x_{2}q_{2} + x_{3}(q_{1}^2 + q_{2}^2)$ ($D_{4}^{+}$-singularity)
	\item $Q(q_{3}, \cdots, q_{n}) + q_{1}^{3}  + q_{2}^3 + x_{1}q_{1} + x_{2}q_{2} + x_{3}q_{1}q_{2}$ ($D_{4}^{-}$-singularity)
	\item $Q(q_{3}, \cdots, q_{n}) + q_{1}^2q_{2} + q_{2}^4 + x_{1}q_{1} + x_{2}q_{2} + x_{3}q_{1}^2 + x_{4}q_{2}^2$ ($D_{5}$-singularity)
\end{enumerate}
\end{teo}

The normal forms of the Lagrangian stable map-germs $\pi \circ L(\tilde{G})$, when $n \leq 4$ are given in table \ref{table2}.
	\begin{table}[!htbp]
	\setlength{\tabcolsep}{10pt}
	\renewcommand{\arraystretch}{1.2}
	\begin{center}
	\begin{tabular}{l l l}
		\hline
		\textbf{$\tilde{G}$ singularity} & \textbf{ $\pi \circ L(\tilde{G})$ singularity} & \textbf{Normal form}  \\ 	\hline
		$A_{2}$                      & Fold                                     &   $q_{1}^2$           \\	\hline
		$A_{3}$                      & Cusp                                   &   $(q_{1}^3 + x_{2}q_{1}, x_{2})$           \\ 	\hline
		$A_{4}$                      & Swallowtail                               &  $( q_{1}^{4} + x_{2}q_{1} + x_{3}q_{1}^2, x_{2}, x_{3})$            \\ 	\hline
		$A_{5}$                      & Butterfly                                 &  $(q_{1}^{5} + x_{2}q_{1} + x_{3}q_{1}^2 + x_{4}q_{1}^3, x_{2}, x_{3}, x_{4})$            \\ 	\hline
		$D_{4}^{-}$                  & Elliptic Umbilic                         &      $(q_{1}^2 - q_{2}^2 + x_{3}q_{1}, q_{1}q_{2} + x_{3}q_{2}, x_{3})$        \\ 	\hline
		$D_{4}^{+}$                  &  Hyperbolic Umbilic                      &        $(q_{1}^2 + x_{3}q_{2}, q_{2}^2 + x_{3}q_{1}, x_{3})$      \\ 	\hline
		$D_{5}$                      &  Parabolic Umbilic                       &    $(q_{1}q_{2} + x_{3}q_{1} , q_{1}^2 + q_{2}^3 + x_{4}q_{2}, x_{3}, x_{4})$ \\ \hline         
	\end{tabular}\caption{ Lagrangian stable singularities}\label{table2}
\end{center}
\end{table}

\begin{obs}\label{obsboardman}\normalfont
	Note in table \ref{table2} that a singularity of $\tilde{G}$ of type $A_{k+1}$ corresponds to a singularity of $\pi \circ L(\tilde{G})$ of type $A_{k}$, $1 \leq k \leq 4$. Hence, if we need to study the singularities of type $A_{k\geq 4}$ of $\tilde{G}$, then we need to study the set $S_{1,1,1}(\pi \circ L(\tilde{G}))$, for instance.
\end{obs}

Now, in order to present a theorem given by Montaldi in \cite{Montaldi}, let $g: Z \rightarrow \R^{n}$ be an immersion, where $Z$ is a smooth manifold, and denote by $\phi_{g}: Z \times \cB \rightarrow \R^k$ the map given by
\begin{align*}
	\phi_{g}(y,u) = F(g(y), u).
\end{align*}
Denote by Imm$(Z, \R^n)$ the subset of $C^{\infty}(Z, \R^n)$ whose elements are proper $C^{\infty}$-
immersions from $Z$ to $\R^n$.
\begin{teo}{\rm{(\cite{Montaldi}, Theorem 1)}}\label{teoMontaldiadap}
	Suppose $F: \R^n \times \cB \rightarrow \R^k$ as above is locally $\cG$-versal. Let $W \subset J^{r}(Z, \R^k)$ be a $\cG$-invariant submanifold, where $Z$ is a $z$-dimensional manifold and let
	\begin{align*}
		R_{W} = \lbrace g \in Imm(Z, \R^n): j^{r}_{1}\phi_{g} \pitchfork  W \rbrace.
	\end{align*}
	Then $R_{W}$ is residual in $Imm(Z, \R^n)$. Moreover, if $\cB$ is compact and $W$ is closed, then $R_{W}$ is open and dense.
\end{teo}

In remark \ref{nicedim} we briefly discuss Mather nice dimensions.

\begin{obs}\label{nicedim}\normalfont
	It is known that an unfolding is $\cG$-versal if and only if it is transversal to the $\cG$-orbits, with this transversality condition being verified at the level of jet spaces. As described in \cite{Livro}, section 4.3, in a series of papers, Mather gave a method to describe generic singularities in terms of transversality, determining a pair of dimensions $(z,k)$ for which the relevant strata of this stratification are the $r$-jets of simple $\cG$-orbits. When $\cG = \cK$ Mather computed in \cite{MatherVI} the codimension $\sigma(z,k)$ of the set of $\cK$-singularities in $J^{r}(Z, \R^k)$. 
	
	Let $g: Z \rightarrow \R^{n}$ be an immersion, thus we have $J^{r}_{1}\phi_{g}: Z \times \cB \rightarrow \R^k$, where $d = \dim \cB$ and $z = \dim Z$. Taking this into account, let $\lbrace W_{1}, W_{2}, \cdots, W_{s} \rbrace$ be the finite set of all the $\cK$-orbits in $J^{r}(Z, \R^k)$ of  $\cK_{e}$-codimension less than $z+d$ and let $\lbrace W_{s+1},\cdots, W_{t}$ $\rbrace$ be the complement of $W_{1} \cup \cdots \cup W_{s}$. If $R$ is the intersection of the residual sets $R_{W_{i}}$, $i=1,\cdots, t$ given in theorem \ref{teoMontaldiadap}, then Mather showed that for a generic $g \in R$, when $\sigma(z,k) + k > z + d$, the associated map $j^{r}_{1}\phi_{g}$ misses the strata $W_{i}$, $i>s$ and is transverse to the strata $W_{i}$ for $i\leq s$. In this case, $\phi_{g}$ is $\cK_{e}$-versal. When $k = 1$ we can consider the group $\cal{R}^{+}$ and the $\cal{R}^{+}$-simple germs coincide with the $\cK$-simple germs. The formulas for the codimension $\sigma(z,k)$ were given in \cite{MatherVI} and can also be found in \cite{Livro}, section 4.3.
\end{obs}

\section{Family of affine distance functions}\label{sec3}
In this section, taking into account \cite{juanjo2}, we study an affine normal plane associated to locally strictly convex surface in $\R^4$ and the family of affine distance functions that arises from this study.  In order to do this, we need first to define locally strictly convex surfaces. Given a surface $M \subset \R^4$ we say that $M$ has \textit{non-degenerate contact} with a hyperplane $\cH$ at $p \in M$ if for any linear function $A: \R^4 \to \R$ such that $\cH = \lbrace x \in \R^4: A(x-p) = 0 \rbrace$, we have that $A_{|_{M}} \to \R$ has a non-degenerate critical point at $p \in M$. Moreover, $\cH$ is said to be a nonsingular support hyperplane of $M$ if $M$ lies on one side of $\cH$, $\cH \cap M = \lbrace p \rbrace$ and $M$ has non-degenerate contact with $\cH$. Using this, we obtain the following definition.

\begin{defi}\normalfont
A surface $M \subset \R^4$ is said to be \textit{strictly convex} if through every point $p \in M$ passes a nonsingular support hyperplane. We say that $M$ is \textit{locally strictly convex} at $p$ if there is a neighborhood $V_{p}$ of $p$ such that $M \cap V_{p}$ is strictly convex.
\end{defi}

Let us take $M \subset \R^4$ a locally strictly convex (at $p \in M$) oriented surface, $\cF = \lbrace X_{1}, X_{2} \rbrace$ a positively oriented frame around $p$ and $\bxi$ a vector field which does not belong to the tangent plane of $M$ at any point, thus we define the metric of $\bxi$, denoted by $g_{\bxi}$, by
\begin{align*}
	g_{\bxi}(Y_{1},Y_{2}) = \dfrac{G_{\cF}(Y_{1}, Y_{2})}{\det G_{\cF}},
\end{align*}
where $G_{\cF}$ is the symmetric bilinear form given by $G_{\cF}(Y_{1}, Y_{2}) = \det \begin{pmatrix} X_{1} & X_{2} & D_{Y_{2}}Y_{1} & \bxi \end{pmatrix}$, $\det G_{\cF} = \det \begin{pmatrix} G_{\cF}(X_{i}, X_{j}) \end{pmatrix}$, $i,j = 1,2$ and $\bxi$ is taken in such a way that $\det G_{\cF}$ is positive definite. In this case, $\bxi$ is said to be a \textit{metric vector field}.

\begin{lema}{\rm{({\cite{juanjo1}, Lemma 3.3})}}
	The symmetric bilinear form $g_{\bxi}$ does not depend on the choice of $\cF$.
\end{lema}

\begin{obs}\normalfont
	As mentioned in \cite{juanjo1}, $g_{\bxi}$ depends only on the class of $\bxi$ in the quotient space $\R^4/T_{p}M$, since $\begin{bmatrix}
		\bxi
	\end{bmatrix} = \begin{bmatrix}
	\bxi_{1}
\end{bmatrix}$ if and only if $\bxi = \bxi_{1} + Z$, where $Z$ is tangent to $M$. Thus, $\det \begin{pmatrix} X_{1} & X_{2} & D_{Y_{2}}Y_{1} & \bxi \end{pmatrix} = \det \begin{pmatrix} X_{1} & X_{2} & D_{Y_{2}}Y_{1} & \bxi_{1} \end{pmatrix}$, from where we get that $g_{\bxi} = g_{\bxi_{1}} = g_{\begin{bmatrix}
	\bxi
\end{bmatrix}}$. If we write 
	\begin{align*}
		\cO_{1} = \lbrace \begin{bmatrix}
			\bxi
		\end{bmatrix} \in \R^4/T_{p}M: g_{\begin{bmatrix}
			\bxi
	\end{bmatrix}}\; \text{is positive definite} \rbrace,
\end{align*}
then $\cO_{1}$ is an open subset of $\R^4/T_{p}M$, since $M$ is locally strictly convex, and the family $\lbrace g_{\begin{bmatrix}
		\bxi
\end{bmatrix}}:  \begin{bmatrix}
\bxi
\end{bmatrix} \in \cO_{1}\rbrace$ is an affine invariant of $M$.

\end{obs}

\begin{teo}{\rm{({\cite{juanjo1}, Theorem 3.7})}}\label{teoframe}
Let $M$ be a locally strictly convex surface and $\bxi$ a metric field. Let
$\cF = \lbrace X_1, X_2 \rbrace$ be a local orthonormal tangent frame relative to $g_{\bxi}$ and let $\sigma$ be an arbitrary transversal plane bundle. Then there exists a unique local frame $\lbrace \bxi_1, \bxi_{2} \rbrace$ of $\sigma$, such that
\begin{align}\label{conditions}
	\det \begin{pmatrix} X_{1} & X_{2} & \bxi_{1} & \bxi_{2} \end{pmatrix} = 1, \ \  h_{1}(X_1, X_1) &= h_{2}(X_1, X_2) = 0, h_{2}(X_1, X_1) = h_{2}(X_2, X_2) =  1\\
	 &\text{and}\; -\bxi_{1} \in \begin{bmatrix}
		\bxi
	\end{bmatrix}.\nonumber
\end{align}
\end{teo}

Let $M$ be a locally strictly convex surface, $\bxi$ a metric field,
$\cF = \lbrace X_1, X_2 \rbrace$ a local orthonormal tangent frame relative to $g_{\bxi}$ and $\sigma$ an arbitrary transversal plane bundle, with associated frame $\lbrace \bxi_1, \bxi_{2} \rbrace$, given by theorem \ref{teoframe}. With this and taking into account notation given in (\ref{decompcampos}), we define the affine normal bundle as follows.

\begin{defi}\normalfont \label{camponormalafim}
The \textit{affine normal plane bundle} of $M$, denoted by $\mathbf{A}$, is the transversal plane bundle generated by $\lbrace \overline{\bxi}_{1}, \overline{\bxi}_{2} \rbrace$, where
\begin{subequations}\label{frameafim}
\begin{align}	
\overline{\bxi}_{1} &= \bxi_{1} - \tau_{1}^{2}(X_1)X_{1} - \tau_{1}^{2}(X_2)X_2,\\
\overline{\bxi}_{2} &= \bxi_{2} - \tau_{2}^{2}(X_1)X_{1} - \tau_{2}^{2}(X_2)X_2.
\end{align}
\end{subequations}

\end{defi}

\begin{obs}\normalfont
	It is possible to show that the affine normal plane bundle, defined in \ref{camponormalafim} does not depend on the transversal plane bundle $\sigma$ or the orthonormal frame $\cF$. See propositions 4.3 and 4.4 in \cite{juanjo1}.
\end{obs}
\begin{prop}\label{propafinenormal}
	Given $M$  a locally strictly convex surface, a metric field $\bxi$ and $\lbrace X_{1}, X_{2}  \rbrace$ a tangent frame orthogonal relative to the metric $g_{\bxi}$, if we take the affine normal bundle $\bA$ as the transversal plane bundle, then the frame given by theorem \ref{teoframe} is exactly the frame $\lbrace \overline{\bxi}_{1}, \overline{\bxi}_{2} \rbrace$ given in (\ref{frameafim}).
\end{prop}
\begin{proof}
	Let $\sigma$ be another transversal vector field and $\lbrace \bxi_{1}, \bxi_{2} \rbrace$ the associated frame given by theorem \ref{teoframe}. Thus,
	\begin{align*}
		\overline{\bxi}_{i} &= \bxi_{i} - \tau_{i}^{2}(X_1)X_{1} - \tau_{i}^{2}(X_2)X_2,\; i=1,2.\\
	\end{align*}
	Thus, $	\det \begin{pmatrix} X_{1} & X_{2} & \overline{\bxi}_{1} & \overline{\bxi}_{2} \end{pmatrix} = 	\det \begin{pmatrix} X_{1} & X_{2} & \bxi_{1} & \bxi_{2} \end{pmatrix} = 1$ and 
	\begin{align*}
		\bxi - \overline{\bxi}_{1} = \left(\bxi - \bxi_{1} \right)	+\tau_{1}^{2}(X_1)X_{1} + \tau_{1}^{2}(X_2)X_2 \in T_{p}M.
	\end{align*}
	Furthermore,
	\begin{align*}
		D_{X}Y &= \overline{\nabla}_{X}Y + \overline{h}_{1}(X,Y)\overline{\bxi}_{1} + \overline{h}_{2}(X,Y)\overline{\bxi}_{2}\\
		&=  \overline{\nabla}_{X}Y + \overline{h}_{1}(X,Y) \left( \bxi_{1} - \tau_{1}^{2}(X_1)X_{1} - \tau_{1}^{2}(X_2)X_2  \right) + \overline{h}_{2}(X,Y) \left( \bxi_{2} - \tau_{2}^{2}(X_1)X_{1} - \tau_{2}^{2}(X_2)X_2 \right) \\
		&= \left[ \overline{\nabla}_{X}Y - \overline{h}_{1}(X,Y) \left(\tau_{1}^{2}(X_1)X_{1} + \tau_{1}^{2}(X_2)X_2  \right) - \overline{h}_{2}(X,Y) \left(\tau_{2}^{2}(X_1)X_{1} + \tau_{2}^{2}(X_2)X_2 \right)   \right] \\ &+ \overline{h}_{1}(X,Y)\bxi_{1} + \overline{h}^{2}(X,Y)\bxi_{2}\\
		&=  {\nabla}_{X}Y + {h}_{1}(X,Y){\bxi}_{1} + {h}_{2}(X,Y){\bxi}_{2}.
	\end{align*}
	Therefore, $ \overline{h}_{i}(X,Y) =  {h}_{i}(X,Y)$, $i=1,2$. From this, we obtain $ h_{1}(X_1, X_1) = h_{2}(X_1, X_2) = 0, h_{2}(X_1, X_1) = h_{2}(X_2, X_2) =  1$ and the result follows.
	
\end{proof}

\begin{obs}\label{remarkafim}\normalfont
	It follows from proposition \ref{propafinenormal} that the bilinear form $\overline{h}_{2}$, associated to $\bA$, is non-degenerate.
\end{obs}

\subsection{Generic singularities of the family of affine distance functions}\label{subsecdist}
From now on, let $M$ be a locally strictly convex surface patch given by $\bx: U \to \R^4$ and $\bxi$ a metric field. Let
$\cF = \lbrace X_1, X_2 \rbrace$ be a local orthonormal tangent frame relative to $g_{\bxi}$ and let $\sigma$ be an arbitrary transversal plane bundle. It follows from theorem \ref{teoframe} that exists a unique local frame $\lbrace \bxi_{1}, \bxi_{2} \rbrace$ of $\sigma$ satisfying the conditions (\ref{conditions}). Taking this into account, we define the family of affine distance functions and focal points of $M$.
\begin{defi}\normalfont\label{ditsafim}
	The \textit{family of affine distance functions} is the $4$-parameter family of functions $\Delta: \R^4 \times M \to \R$ defined as follows:	if $p \in \R^4 $ and $\bx(u) \in M$, $\Delta(p,u)$ is such that
		\begin{align*}
			\bx(u) - p = Z(p,u) + \Delta(p,u)\bxi_{2},
	\end{align*}
where $Z(p,u) \in  T_{\bx(u)}M \oplus \bxi_{1}$.	
\end{defi}

\begin{obs}\normalfont
	It is possible to show that the family of affine distance functions, defined in \ref{ditsafim} does not depend on the transversal plane bundle $\sigma$ or the orthonormal frame $\cF$. See proposition 5.2 in \cite{juanjo2}.
\end{obs}

\begin{teo}{\rm{({\cite{juanjo2}, Theorem 5.3})}}\label{teojuanjo2}
	The affine distance function $\Delta_{p}$ has a singularity at $u$ if and
	only if $\bx(u) - p$ belongs to the affine normal plane $\bA_{\bx(u)}$.
\end{teo}

\begin{obs}\normalfont \label{obscat}
	It follows from theorem \ref{teojuanjo2} that the catastrophe set of $\Delta$ is given by
	\begin{align*}
		C_{\Delta} = \lbrace (u,p): p = \bx(u) - \lambda \bv,\; \text{where $\bv \in A_{\bx(u)}, \lambda \in \R$} \rbrace.
	\end{align*}
\end{obs}

Next, in order to define focal points of locally strictly convex surfaces in $\R^4$, we first present the notions of $\bv$-principal curvature and $\bv$-principal directions, where $\bv$ is a vector taken on the affine normal plane. This notion was given in \cite{juanjo2}. Given $\bv \in \bA$, we say that a tangent direction of $M$ is a \textit{$\bv$-principal direction} if it is an eigenvector of the shape operator $S_{\bv}$, associated to $\bv$, while the \textit{$\bv$-principal curvatures} are the eigenvalues of $S_{\bv}$.

\begin{defi}\normalfont
	A point $p = \bx(q) + t \bv$ where $\bv \in \bA$ is an \textit{affine focal point} of $M$ at $q$ if
	$t \neq 0$ and $\frac{1}{t}$ is a $\bv$-principal curvature.
\end{defi}

\begin{teo}{\rm{({\cite{juanjo2}, Theorem 6.2})}}
	The affine distance function  $\Delta_{p}: U \to \R$ has a degenerate critical point at $q \in U$ if and only if $p$ is an affine focal point of $M$ at $q$.
\end{teo}

One can relate the affine focal points of $M$ with the critical points of a line congruence $F(u,t) = \bx(u) + t \bv (u)$, where $\bv : U \to \R^4 \setminus \lbrace \bzero \rbrace$ is a vector field on $\bA$.

\begin{prop}
	A point $p = \bx(u_0) + t_0 \bv$, where $\bv \in \bA$, is an affine focal point of $M$ if and only if $(u_0, t_0)$ is a critical point of $F_{(\bx, \bv)}(u,t) = \bx(u) + t \bv (u)$, that is, a critical point of the line congruence associated to $\bx$ and $\bv$.
\end{prop}

\begin{proof}
Let $\bw$ be a vector in $T_{\bx(u_0)}M$ and $a \in \R$. Note that
\begin{align*}
	(F_{(\bx, \bv)})_{*}(\bw, 0) &= D_{\bw}(\bx + t \bv) = \bw + t\left( -S_{\bv}(\bw) + \nabla^{\perp}_{\bw}\bv \right) = \left( I_{\R^2} - tS_{\bv} \right)\bw +  \nabla^{\perp}_{\bw}\bv,\\
	(F_{(\bx, \bv)})_{*}(0, a) &= a \bv,
\end{align*}
where  $\nabla^{\perp}_{\bw}\bv$ denotes the component of $D_{\bw}\bv$ in the direction $\bv$. Thus, $(F_{\bx, \bv})_{*}(\bw, a) = 0$ if and only if $a = -\nabla^{\perp}_{\bw}\bv$ and $\left( I_{\R^2} - tS_{\bv} \right)\bw = 0$. Therefore, if $(u_0, t_0)$ is a critical point of $F_{(\bx, \bv)}$, then $\frac{1}{t_0}$ (it is easy to see that $t_0 \neq 0$) is a $\bv$-principal affine curvature. On the other hand, if $\frac{1}{t_0}$ is a $\bv$-principal affine curvature it follows that for $a =- (\nabla^{\perp}_{\bw}\bv)_{u_0}$ we have $(F_{\bx, \bv})_{*}(u_0, t_{0})(\bw, a) = 0$, thus $(u_0, t_{0})$ is a critical point.
\end{proof}

Let us take a fixed metric field $\bxi$ and the frame $\cF_{1} = \lbrace \overline{\bxi}_{1}, \overline{\bxi}_{2} \rbrace$ as the frame that generates the transversal plane bundle $\mathbf{A}$, obtained from the unique local frame $\cF = \lbrace \bxi_{1}, \bxi_{2} \rbrace$ of $\sigma$ satisfying the conditions (\ref{conditions}).

\begin{defi}\normalfont
	The map $\cW: U \to \R_{4}$, where $\R_{4}$ is the dual space of the underlying
	vector space for $\R^4$, such that for all $u \in U$
	\begin{align}
		\cW(u)(\bv) &= 0,\; \text{for all $\bv \in T_{\bx(u)}M \oplus \overline{\bxi}_{1}(u)$},\\
		\cW(u)(\overline{\bxi}_{2}(u)) &= 1
	\end{align}
is called the \textit{conormal map} of $\overline{\bxi}_{2}$.
\end{defi}

\begin{obs}\label{observid}\normalfont

Let $\bxi$ be a fixed metric field and
	\begin{align*}
Emb_{lsc}(U, \R^4) = \left\lbrace (\bx, \bxi): \bx: U \rightarrow \R^4\; \text{is a locally strictly convex embedding and $\bxi$ is a metric field } \right\rbrace
\end{align*}
the space of locally strictly convex surfaces with the Whitney $C^{\infty}$-topology. With notation as above, let us define the following space
	\begin{multline*}
	S(U, \R^4 \times \R^4 \setminus \lbrace \mathbf{0} \rbrace) = \left\lbrace (\bx, \bxi, \cW) \in C^{\infty}(U, \R^4 \times \R^4 \setminus \lbrace \bzero \rbrace): \bx\in Emb_{lsc}(U, \R^4) \right.\\
	\left. \text{and}\; \cW\; \text{is}\; \text{the conormal of $\bx$ relative to $\overline{\bxi}_{2}$} \right\rbrace, 
\end{multline*}
where the conormal is taken as a vector field. Then, we identify (with the Whitney $C^{\infty}$-topology) the spaces $Emb_{lsc}(U, \R^4)$ and $S(U, \R^4 \times \R^4 \setminus \lbrace \mathbf{0} \rbrace)$.

\end{obs}

\begin{prop}\label{lemaconormal}
	The conormal map $\cW$ is an immersion.
\end{prop}

\begin{proof}
Since $\cW(u)(\bv) = 0$, for all $\bv \in T_{\bx(u)}M \oplus \overline{\bxi}_{1}(u)$, it follows that
\begin{align*}
	0 = Y(\cW(\bv)) &= D_{Y}\cW(\bv) + \cW(D_{Y}\bv)\\
	&= \cW_{\ast}(Y)(\bv) + \cW(\overline{h}_{2}(\bv, Y)\overline{\bxi}_{2})\\
	&= \cW_{\ast}(Y)(\bv) + \overline{h}_{2}(\bv, Y).
\end{align*}
Hence, $\cW_{\ast}(Y)(\bv) = - \overline{h}_{2}(\bv, Y)$, that is non-degenerate via remark \ref{remarkafim}. Thus, the result follows.
\end{proof}

\begin{prop}\label{propmorse}
	The family of affine distance functions $\Delta$ is a Morse family of functions.
\end{prop}

\begin{proof}
In order to prove that $\Delta$ is a Morse family of functions, we need to prove that the map $\lambda_{\Delta}(u,p) = \left( \dfrac{\partial \Delta}{\partial u_{1}},\dfrac{\partial \Delta}{\partial u_{2}} \right)$ is not singular. Note that $\Delta(u,p) = \cW(u)(\bx(u) - p)$, hence $\dfrac{\partial \Delta}{\partial p_{i}} = -\cW_{i}$ and $\dfrac{\partial^{2}\Delta}{\partial p_{i}\partial u_{j}} = -\dfrac{\partial \cW_{i}}{\partial u_{j}}$, where $j=1,2$, $\cW_{i} = \cW(u)(e_{i})$, $i=1,2,3,4$, $j=1,2$ and $\lbrace e_{1}, e_{2}, e_{3}, e_{4} \rbrace$ is the canonical basis. Furthermore, the Jacobian matrix of $\lambda_{\Delta}$ is given by

\begin{align*}
J(\lambda_{\Delta}) =	\begin{bmatrix}
		\dfrac{\partial^{2}\Delta}{\partial u_{1}^{2}} & \dfrac{\partial^{2}\Delta}{\partial u_{1} \partial u_{2}} & -\dfrac{\partial \cW_{1}}{\partial u_{1}} & -\dfrac{\partial \cW_{2}}{\partial u_{1}} & -\dfrac{\partial \cW_{3}}{\partial u_{1}}&  \dfrac{\partial \cW_{4}}{\partial u_{1}}\\
		\dfrac{\partial^{2}\Delta}{\partial u_{2} \partial u_{1}} &	\dfrac{\partial^{2}\Delta}{\partial u_{2}^{2}} & -\dfrac{\partial \cW_{1}}{\partial u_{2}} & -\dfrac{\partial \cW_{2}}{\partial u_{2}} & -\dfrac{\partial \cW_{3}}{\partial u_{2}}&  -\dfrac{\partial \cW_{4}}{\partial u_{2}}.
	\end{bmatrix}
\end{align*}
It follows from proposition \ref{lemaconormal} that $\cW$ is an immersion, then $J(\lambda_{\Delta})$ has $\rank$ $2$ and $\lambda_{\Delta}$ is not singular.
\end{proof}

\begin{obs}\label{observid2}\normalfont
Note that the family of affine distance functions satisfies $\Delta(p, u) = \langle \bx(u) - p, \cW \rangle $, where the conormal is taken as a vector field. Let us define the following maps
\begin{align}
	H: \left( \R^{4}  \times \R^4 \setminus \lbrace \bzero \rbrace \right) \times \R^4 &\rightarrow \R\\
	(A, B, C, D) &\mapsto \langle B, A-D \rangle \nonumber
\end{align}
\begin{align}
	g: U &\rightarrow  \R^{4}  \times \R^4 \setminus \lbrace \bzero \rbrace\\
	u &\mapsto (\bx(u), \bxi(u), \cW(u)),\nonumber
\end{align}
where $g \in S(U, \R^{4}  \times \R^4 \times \R^4 \setminus \lbrace \bzero \rbrace)$.
If we fix a parameter $C$,  $H_{C}: \R^{4}  \times \R^4 \setminus \lbrace \bzero \rbrace \rightarrow \R$ is a submersion, therefore, $H_{C} \circ g$ is a contact map. 
Finally, note that
 \begin{align*}
 	\Delta(u,p) = H \circ \left(g, Id\big|_{\R^4} \right)(u,p).
\end{align*}
\end{obs}
Next, taking into account theorem \ref{teoMontaldiadap}, we show that for a generic element in $Emb_{lsc}(U, \R^4)$ the family $\Delta$ is versal. 

\begin{prop}\label{Prop5.4}
	For a residual subset of $Emb_{lsc}(U, \R^4)$ the family $\Delta$ is locally 
	$\cP$-$\mathcal{R}^{+}$-versal.
\end{prop}
\begin{proof}\,
	Since the family $H$ given in remark \ref{observid2} is a family of submersions, following remarks {\ref{observid}} and \ref{observid2} we can apply theorem \ref{teoMontaldiadap} in order to show that there is a residual subset of $Emb_{lsc}(U, \R^4)$ for which $\Delta$ is locally $\mathcal{P}$- ${\cal{R}}^{+}$-versal.
\end{proof}

Furthermore, taking into account proposition \ref{Prop5.4} we obtain the generic singularities of $\Delta_{p}$.
\begin{teo} \label{classifsingdelta}
There is a residual subset  $ \mathcal{V} \subset Emb_{lsc}(U, \R^4)$ such that: %for any $\bx \in \mathcal{V}$, $M = \bx(U)$ has the following properties:
\begin{enumerate}
	\item The singularities of $\Delta_{p}$ are of type $A_{1}, A_{2}, A_{3}, A_4, A_5, D_4$ or $D_5$.
	\item The singularities of $\Delta_{p}$ are $\mathcal{R}^{+}$-versally unfolded by the family $\Delta$.
\end{enumerate}
\end{teo}

\begin{proof}
Taking into account the nice dimensions (see remark \ref{nicedim}) we have that $m=2$, $d=4$, $k=1$ and $m+d = 6 < \sigma(2,1) = 9$. With this, the result follows from theorem \ref{teoMontaldiadap} and remark \ref{nicedim}.
\end{proof}

Now we show an expression for the Hessian of the affine distance function $\Delta_{p}$ at a critical point.
\begin{lema}\label{carahessiana}
	Let $\bx: U \to \R^4$ be localy strictly convex surface, with $\bx(U) = M$. If $u$ is a critical point of $\Delta_{p}$, then the Hessian of $\Delta_{p}$ at $u$ has the form
	\begin{align*}
	H(\bX,\bY) &= h^{2}(\bX,  (-Id + \lambda S_{\bnu})(\bY)),\; \text{where $\bnu \in \bA_{\bx(u)}$ and $\bX, \bY \in T_{\bx(u)}M$}.
\end{align*}
Furthermore, $H$ degenerate if and only if $\det \begin{pmatrix}
	(-Id + \lambda S_{\bnu})
\end{pmatrix} = 0$. 
\end{lema}

\begin{proof}
	First, let us consider a critical point $u$ of $\Delta_{p}$. Since $\bx(u) - p = \mathbf{Z} + \delta \bxi_{1} + \Delta_{p}\bxi_{2}$, where $\mathbf{Z}$ is tangent we obtain, for a tangent vector $\bX$,
%	\begin{align*}
%		\bX = D_{\bX}Z + \bv\left(\Delta_{p}\right)\bxi_{2} + \Delta_{p}D_{\bX}\bxi_{2},
%	\end{align*}
	%where $\bX$ is a tangent vector. Thus
	\begin{align}\label{aboveexp}
		\bX &= \nabla_{\bX}\mathbf{Z} + h^{1}(\bX, \mathbf{Z})\bxi_{1} +  h^{2}(\bX, \mathbf{Z})\bxi_{2} + \delta D_{\bX}\bxi_{1} + X(\delta)\bxi_{1}  + \bX\left(\Delta_{p}\right)\bxi_{2} + \Delta_{p}D_{\bX}\bxi_{2}.
	\end{align}
	Taking into account that $\bX$ is tangent we conclude from (\ref{aboveexp}) that
	\begin{align}
		\nabla_{\bX}\mathbf{Z} &= \bX - \delta S_{\bxi_1}(\bX) - \Delta_{p}S_{\bxi_{2}}(\bX)\label{eq1hess},\\
		\bX(\Delta_{p}) &= -h^{2}(\bX, \mathbf{Z}).\label{eq2hess}
	\end{align}	
	Since we are considering a critical point, it follows from $h^{2}$ non-degenerate that $\mathbf{Z}=0$. Thus, the Hessian of $\Delta_{p}$ is given by
	\begin{align}\label{expressaohessiana}
		H(\bX,\bY) = - \bY\left(h^{2}(\bX, \mathbf{Z})\right).
	\end{align}
	It is known that
	\begin{align*}
		\nabla_{\bY}h^{2}(\bX, \mathbf{Z}) = \bY\left(h^{2}(\bX, \mathbf{Z})\right) - h^{2}(\nabla_{\bY}\bX, \mathbf{Z}) - h^{2}(\bX, \nabla_{\bY}\mathbf{Z}),
	\end{align*}
	thus, since $Z=0$, in (\ref{expressaohessiana}) we have that
	\begin{align}\label{expressaohessiana22}
		H(\bX,\bY) &= - h^{2}(\bX, \nabla_{\bY}\mathbf{Z})\nonumber\\
		&= -h^{2}(\bX,  \bY - \delta S_{\bxi_1}(\bY) - \Delta_{p}S_{\bxi_{2}}(\bY)),\;\text{from (\ref{eq1hess})}
	\end{align}
	From the fact that we are at a critical point of $\Delta_{p}$, we obtain from theorem \ref{teojuanjo2} that $p = \bx(u) - \lambda \bnu$, where $\bnu \in \bA_{\bx(u)}$, which implies that $\lambda\bnu = \delta \bxi_{1} + \Delta_{p}\bxi_{2}$. Thus, taking the shape operators, we have that
	\begin{align*}
			\lambda S_{\bnu} = \delta S_{\bxi_{1}} + \Delta_{p} S_{\bxi_{2}}.
	\end{align*}
	Hence, from (\ref{expressaohessiana22} we conclude that
	\begin{align*}
		H(\bX,\bY) &= -h^{2}(\bX,  \bY - \lambda S_{\bnu}(\bY))\\
		&= h^{2}(\bX,  (-Id + \lambda S_{\bnu})(\bY)).
	\end{align*}
	It follws from the fact that $h^2$ is non-degenerate that $H$ degenerate if and only if 
	\begin{align*}\det \begin{pmatrix}
		(-Id + \lambda S_{\bnu})
	\end{pmatrix} = 0.\end{align*}
\end{proof}

\begin{obs}\normalfont
Since the bifurcation set of the germ of family of affine distance functions $\Delta$ is given by
\begin{align*}
	B_{\Delta} = \lbrace p \in (\R^4, 0):\; \exists \; (u,p) \in C_{\Delta}, \text{such that $\rank$Hess$(\Delta)< 2$}\rbrace,
\end{align*}
it follows from lemma \ref{carahessiana} that the points in $B_{\Delta}$ are exactly the affine focals of $M$. 
\end{obs}

Next, we define affine semiumbilic point, taking into account definition 7.1 in \cite{juanjo2}, in order to provide a characterization of this points in terms of the singularities of the affine distance function.

\begin{defi}\normalfont \label{semiumbilicpoint}
A point $q = \bx(u)$ is said to be \textit{affine semiumbilic} if there exists a nonzero normal vector $\bnu \in \bA_{\bx(u)}$ such that $S_{\bnu}$ is a multiple of the identity.
\end{defi}

\begin{teo}\label{caracsemi}
	A point is affine semiumbilic if and only if it is a singularity of corank $2$ of some affine distance function $\Delta_{p}$, i.e. the hessian matrix of $\Delta_{p}$ has $\rank$ $0$.
\end{teo}

\begin{proof}
It follows from lemma \ref{carahessiana} that the hessian of $\Delta_{p}$ is given by
\begin{align}\label{expressaohessiana2}
	H(\bX,\bY) &= h^{2}(\bX,  (-Id + \lambda S_{\bnu})(\bY)),
\end{align}
where $\bnu \in \bA_{\bx(u)}$. If we suppose that $\Delta_{p}$ has a singularity of corank $2$, then $H = 0$, therefore $-Id + \lambda S_{\bnu} = 0$, since $h^{2}$ is non-degenerate. From this we obtain that $ S_{\bnu} = \dfrac{1}{\lambda} Id$, that is, $\bx(u) = q$ is a semiumbilic point of $M$. On the other hand, if $\bx(u) = q$ is a semiumbilic point, then $ S_{\bnu} = \dfrac{1}{\lambda} Id$, for some $\bnu \in A_{q}$. Thus, if we take $p = \bx(u) -\lambda \bnu$, it follows from theorem \ref{teojuanjo2} that $u$ is a critical point of $\Delta_{p}$ and from lemma \ref{carahessiana} we have that the Hessian of $\Delta_{p}$ is given by
\begin{align*}
H(\bX,\bY)=	h^{2}(\bX,  (-Id + \lambda S_{\bnu})(\bY)).
\end{align*} 
From the fact that $S_{\bnu} = \dfrac{1}{\lambda} Id$ we conclude that $H=0$, that is, $q$ is a corank $2$ singularity of $\Delta_{p}$.

\end{proof}

\begin{defi}\normalfont
	Let $\bx: U \to \R^4$, $\bx(U) = M$, be a locally strictly convex patch surface. A point $\bx(u) = q \in M$ which is a singular point of type $A_{k}$, $k \geq 4$, for some affine distance function $\Delta_{p}$ is said to be a \textit{ridge point of order $k$}. The \textit{ridge of order $k$} is the set of all $k$-order ridge points for a given $k \geq 4$.
\end{defi}

\begin{prop}\label{propridge}
	Let $\bx: U \to \R^4$, $\bx(U) = M$, be a locally strictly convex patch surface that is affine support function generic. The
	set of $4$-order ridge points is either empty or is a smooth curve.
	If exist, the ridge points of order $5$ are isolated points on the curve of $4$-order ridge points.
\end{prop}

\begin{proof}
	The proof follows the same idea of proposition 7.21 in \cite{Livro}, we just need to use the Lagrangian map associated to $\Delta$ instead of the normal exponential map of $M$. The Lagrangian map is denoted by $F_{(\bx, \bA)}: (C_{\Delta}, 0) \to \R^4$, where $		C_{\Delta} = \lbrace (u,p): p = \bx(u) - \lambda \bv,\; \text{where}\; \bv \in A_{\bx(u)}, \lambda \in \R \rbrace$, $\bA_{\bx(u)}$ denotes the affine normal plane and $F_{(\bx, \bA)}(u,p) = p$. Note that $C_{\Delta}$ coincides with the affine normal plane bundle of $M$. Then, the ridge of order $4$ is given by the projection of the regular curve $S_{1,1,1}(F_{(\bx, \bA)})$, which lies in the affine normal plane bundle of $M$, by the canonical projection $\pi: C_{\Delta} \to M$. As shown in \cite{Livro}, generically, the kernel of this projection is transversal to this curve at every point, therefore, its image in $M$ is also a regular curve. Furthermore, points of order $5$ are generically isolated in $S_{1,1,1}(F_{(\bx, \bA)})$.
\end{proof}

\begin{obs}\normalfont
	In the proof of proposition \ref{propridge} we denote the Lagrangian map associated to the family of affine distance functions $\Delta$ by $F_{(\bx, \bA)}$ because this is how we denote  in next section the affine normal plane congruence associated to the affine normal plane bundle, that is the Lagrangian map of $\Delta$. 
\end{obs}

%\begin{defi}\normalfont
%	The \textit{rib of order k} of a locally strictly convex surface $M \subset \R^4$ is the set of points of $\R^4$ that determine affine distance functions with corank $1$ singularities of type $A_{k}$, $k \geq 3$.
%\end{defi}

\section{Singularities of plane congruences}	\label{sec4}	
In this section we study the generic singularities of plane congruences. We first discuss the general case, in which a surface in $\R^4$ and a family os transversal planes are considered and then we discuss a special case of plane congruences, the case of congruences generated by a locally strictly convex surface $M \subset \R^4$ and its affine normal plane (in the sense of \cite{juanjo2}).
\subsection{General case}
In order to work with this case, we adapt the approach applied in \cite{Izumiya} and \cite{nossoartigo} to the codimension 2 case.
\begin{defi}\normalfont
	A 2-parameter plane congruence in $\R^4$ is a 2-parameter family of planes in $\R^4$. Locally, we write 
	\begin{align*}
		F_{\left(\bx, \bxi, \bd  \right)}: U \times I \times J &\rightarrow \R^4\\
		(u,t,l) &\mapsto \bx(u) + t\bxi(u) + l\bd(u),
	\end{align*}
	where
	\begin{itemize}
		\item $\bx: U \rightarrow \R^4 $ is smooth and it is called a reference surface of the congruence.
		\item $\bxi, \bd: U \rightarrow \R^4 \setminus \lbrace \bm{0} \rbrace $ are smooth and linearly independent. We call $\bxi$ and $\bd$ director surfaces of the congruence.
	\end{itemize}
	Let us write
	\begin{align}\label{setli}
		{\cal{L}} = \lbrace (\bxi, \bd) \in C^{\infty}(U, \R^4 \setminus \lbrace \bm{0} \rbrace \times \R^4 \setminus \lbrace \bm{0} \rbrace): \bxi\; \text{and}\; \bd\; \text{are linenarly independent} \rbrace
	\end{align}
\end{defi}

\begin{lema}
	The singular points of a plane congruence $F_{\left(\bx, \bxi, \bd \right)}$ are the points $(u, t, l)$ such that
	\begin{align}\label{eqpontosing}
		&t^2 \langle \bxi_{u_1} \wedge \bxi_{u_2} \wedge \bxi, \bd \rangle + l^2 \langle \bd_{u_1} \wedge \bd_{u_2} \wedge \bxi, \bd \rangle + tl \left ( \langle \bxi_{u_1} \wedge \bd_{u_2} \wedge \bxi, \bd \rangle + \langle \bd_{u_1} \wedge \bxi_{u_2} \wedge \bxi, \bd \rangle \right) \\
		&+ t\left( \langle \bxi_{u_1} \wedge \bx_{u_{2}} \wedge \bxi, \bd  \rangle + \langle \bx_{u_1} \wedge \bxi_{u_{2}} \wedge \bxi, \bd  \rangle  \right) + l \left( \langle \bd_{u_1} \wedge \bx_{u_{2}} \wedge \bxi, \bd  \rangle + \langle \bx_{u_1} \wedge \bd_{u_{2}} \wedge \bxi, \bd  \rangle \right) \nonumber \\
		&+ \langle \bx_{u_1} \wedge \bx_{u_{2}} \bxi, \bd \rangle = 0. \nonumber
	\end{align}
\end{lema}	

\begin{proof}
	We know that $(u,t,l) \in U \times I \times J$ is a singular point of $F_{\left( \bx, \bxi, \bd \right)}$ if and only if $\det JF(u,t,l) = 0$, where
	\begin{align*}
		JF(u,t,l)	=  \begin{bmatrix}
			\bx_{u_{1}} + t\bxi_{u_1} + l\bd_{u_1} & \bx_{u_{2}} + t\bxi_{u_2} + l\bd_{u_2} & \bxi & \bd
		\end{bmatrix}_{4\times 4}
	\end{align*}
	is the jacobian matrix of $F$ at $(u,t,l)$.	From 
	\begin{align*}
		\det JF(u,t,l) = \langle (\bx_{u_{1}} + t\bxi_{u_1} + l\bd_{u_1}) \wedge (\bx_{u_{2}} + t\bxi_{u_2} + l\bd_{u_2}) \wedge \bxi, \bd  \rangle = 0
	\end{align*}
	we obtain (\ref{eqpontosing}).
\end{proof}

Let us take $F_{\left(\bx, \bxi, \bd  \right)}$ a congruence of planes. Let us consider $F_{i}$, $\bx_{i}$,  $\bxi_{i}$ and $\bd_{i}$, $i=1,2,3,4$ the coordinate functions of $F_{\left(\bx, \bxi, \bd  \right)}$, $\bx$, $\bxi$ and $\bd$, respectively.

Given $(u_{0}, t_{0}, l_{0}) \in U \times I$, it follows from the fact that $\bxi$ and $\bd$ are linearly independent that at least one of the $2\times 2$ minors
\begin{align*}
	M_{\bxi, \bd} = \begin{bmatrix}
		\bxi_{1} & \bd_{1}\\
		\bxi_{2} & \bd_{2}\\
		\bxi_{3} & \bd_{3}\\
		\bxi_{4} & \bd_{4}
	\end{bmatrix}
\end{align*}	
is different than zero at $(u_{0}, t_{0}, l_{0})$. Let us suppose, without loss of generality		
\begin{align}\label{deteq}
	\det \begin{bmatrix}
		\bxi_{1} & \bd_{1}\\
		\bxi_{2} & \bd_{2}\\
	\end{bmatrix} \neq 0\, \text{at}\, u_{0}.
\end{align}
Then there is an open subset $U_{1,2} \subset U$ such that $\bxi_{1}\bd_{2} - \bxi_{2}\bd_{1}(u) \neq 0 $, for all $u \in U_{1,2}$.
By considering $\bxi_{1}(u_0) \neq 0$ (the case for $\bxi_{2}(u_{0}) \neq 0$ is analogous), there is a germ of diffeomorphism $h: (\R^4, u_{0}, 0,0) \rightarrow (\R^4,  u_{0}, t_{0}, l_{0})$, given by
{\footnotesize \begin{align*}
		h(u,t,l) = \left(u_{1}, u_{2}, \dfrac{t \bxi_{1}  + a_{0} - \bx_{1}}{\bxi_{1}} - \dfrac{\bd_{1}}{\bxi_{1}} \left(\dfrac{b_{0}\bxi_{1} - \bx_{2}\bxi_{1} + \bx_{1}\bxi_{2} - a_{0}\bxi_{2}}{\bxi_{1}\bd_{2} - \bxi_{2}\bd_{1}}\right), l + \dfrac{b_{0}\bxi_{1} - \bx_{2}\bxi_{1} + \bx_{1}\bxi_{2} - a_{0}\bxi_{2}}{\bxi_{1}\bd_{2} - \bxi_{2}\bd_{1}}  \right),
\end{align*}}	
where $a_{0} = t_{0}\bxi_{1} + l_{0}\bd_{1} + \bx_{1}(u_{0})$ and $b_{0} = t_{0}\bxi_{2} + l_{0}\bd_{2} + \bx_{2}(u_{0})$, we obtain $\widetilde{F}_{\left(\bx, \bxi, \bd  \right)} = F_{\left(\bx, \bxi, \bd  \right)} \circ h$ such that its first and second coordinates are given by
\begin{subequations}\label{EqFtilda}
\begin{align}
	\widetilde{F}_{1} &= a_{0} + t\bxi_{1} + l\bd_{1}\\
	\widetilde{F}_{2} &= b_{0} + t\bxi_{2} + l\bd_{2}.
\end{align}	
\end{subequations}	
If we write $G(u,t,l) = (\widetilde{F}_{1}, \widetilde{F}_{1})$, then $G^{-1}(a_0,b_0) = \lbrace (u,0,0): u \in U_{1,2} \rbrace$. Let us take $\pi_{1,2}: \R^4 \rightarrow \R^2$, given by $\pi_{1,2}(x_{1}, x_{2}, x_{3}, x_{4}) = (x_{3}, x_{4})$ and
\begin{align*}
	\widetilde{f}(u) = \pi_{1,2} \circ \widetilde{F}_{\left(\bx, \bxi, \bd  \right)}(u,0,0).
\end{align*}
\begin{prop}\label{propunfolding}
	With notation as above, the map germ $\widetilde{F}: (\R^4, u_{0}, 0,0) \rightarrow \R^4$ is a two dimensional unfolding of $\widetilde{f}(u) = \pi_{1,2} \circ \widetilde{F}(u,0,0)$.
\end{prop}		
\begin{proof}
	Let us write $G(u,t,l) = (\widetilde{F}_{1}, \widetilde{F}_{2})$, then it follows from (\ref{deteq}) that $\dfrac{\partial G}{\partial r }$, where $r = (t,l)$, is non-singular and $G^{-1}(a_0,b_0) = \lbrace (u,0,0): u \in (\R^2, u_{0}) \rbrace$. If we take
	\begin{align*}
		i: (\R^2, u_{0}) &\rightarrow \R^4\\
		u &\mapsto (u,0,0)\\
		j:\R^2 &\rightarrow \R^4\\
		y &\mapsto (a_0,b_0,y)\\
	\end{align*}
	then, $\widetilde{F} \circ i = j \circ \widetilde{f}$. Note that $\widetilde{F}$ is transverse to $j$, since we are considering $\widetilde{F}_{1}$ and $\widetilde{F}_{2}$ written as in (\ref{EqFtilda}). Furthermore, 
	\begin{align*}
		\lbrace (u,t,l,y): \widetilde{F}(u,t,l) = j(y) \rbrace = \lbrace  (u,0,0, \widetilde{f}(u)): u \in (\R^2, u_{0}) \rbrace.
	\end{align*}
	Notice that $(i,\widetilde{f}):  (\R^2, u_{0}) \rightarrow \lbrace (u,t,l,y): \widetilde{F}(u,t,l) = j(y) \rbrace$ is given by $(i,\widetilde{f})(u) = (u,0,0,\widetilde{f}(u))$, thus is a diffeomorphism. Therefore, $(\widetilde{F}, i, j)$ is a two dimensional unfolding of $\widetilde{f}$, via definition \ref{defunfolding}.
\end{proof}

\begin{lema}\label{Lema4.1}
	Let $W \subset J^{k}(2,2)$ be a submanifold. For any fixed map germs  $\bxi, \bd: U \rightarrow \R^{4} \setminus \lbrace \bm{0} \rbrace$ and any fixed point $q_{0} = (u_{0}, t_{0}, l_{0}) \in U \times I \times J$ with
	\begin{align*}
		\det \begin{bmatrix}
			\bxi_{1} & \bd_{1}\\
			\bxi_{2} & \bd_{2}\\
		\end{bmatrix} \neq 0
	\end{align*}
	the set 
	\begin{align*}
		T^{(\bxi, \bd)}_{1,2,W, q_{0}} = \left\lbrace \bx \in C^{\infty}(U, \R^4): j^{k}_{1} \left( \pi_{1,2} \circ \widetilde{F}_{(\bx,\bxi, \bd)} \right) \pitchfork W\; \text{at}\; (u_{0},t_0,l_0)  \right\rbrace
	\end{align*}
	is a residual subset of $C^{\infty}\left( U, \R^{4} \right)$.
\end{lema}

\begin{proof}
%	In which follows we identify $C^{\infty}(U, \R^4) \times C^{\infty}(U, \R^4 \setminus \lbrace \bzero \rbrace) \times C^{\infty}(U, \R^4 \setminus \lbrace \bzero \rbrace)$ with $C^{\infty}(U, \R^4 \times \R^4 \setminus \lbrace \bzero \rbrace \times \R^4 \setminus \lbrace \bzero \rbrace)$ and we take the $C^{\infty}$-Whitney topology induced on $C^{\infty}(U, \R^4) \times \lbrace \bxi \rbrace \times \lbrace \bd \rbrace$.
	
	Let us take $\lbrace C_{j} \rbrace_{j=1}^{\infty}$ a countable open cover for $W$, such that $\overline{C_{j}}$ is  compact for all $j \in \NN$. Define
	\begin{align}\label{specialset}
		T^{(\bxi, \bd)}_{1,2,W, q_{0}, C_{j}} = \lbrace \bx \in C^{\infty}(U, \R^4): j_{1}^{k}\left( \pi_{1,2} \circ \tF_{(\bx,\bxi, \bd)} \right) \pitchfork W\;  \text{with}\; j^{k}_{1}\left( \pi_{1,2} \circ \widetilde{F}_{(\bx,\bxi, \bd)} \right)(q_{0}) \in \overline{C_{j}}\rbrace.
	\end{align}
	We claim that \ref{specialset} is open. In fact, taking into account that the map	$\hat{j}^{k}_{1}: C^{\infty}(U_{1,2},  \R^4) \to C^{\infty}(U_{1,2} \times I \times J, J^{k}(2,2))$, given by $\hat{j}^{k}_{1}(\bx) = j_{1}^{k}\left( \pi_{1,2} \circ \tF_{(\bx,\bxi, \bd)} \right)$ is continuous, define
	\begin{align*}
		O_{W, C_{j}} = \lbrace g \in C^{\infty}(U_{1,2} \times I \times J, J^{k}(2,2)): g \pitchfork W\; \text{at}\; q_{0}, g(q_0) \in \overline{C_{j}} \rbrace,
	\end{align*}
	which is open (see chapter 2, proposition 4.5 in \cite{gg}). Thus, as the restriction map $res_{|_{U_{1,2}}}: C^{\infty}(U,\R^4) \to C^{\infty}(U_{1,2}, \R^4)$ is also continuous, it follows that
	\begin{align*}
		T^{(\bxi, \bd)}_{1,2,W,q_{0}, C_{j}} = \left(res_{|_{U_{1,2}}}  \right)^{-1}\left(\left( \hat{j}_{1}^{k} \right)^{-1}(O_{W, C_{j}})\right)\; \text{is open}.
	\end{align*}
	If we show that $T^{(\bxi, \bd)}_{1,2,W,q_{0}, C_{j}}$ is dense, then $T^{(\bxi, \bd)}_{1,2,W, q_{0}} = \bigcap\limits_{j \in \NN}T^{(\bxi, \bd)}_{1,2,W,q_{0}, C_{j}}$ is residual. Since the restriction map is surjective it is enough to show that
\begin{align*}
	T^{(\bxi, \bd)}_{1,2,W,q_{0}, C_{j}, U_{1,2}} = \lbrace \bx \in C^{\infty}(U_{1,2}, \R^n): j_{1}^{k}\left( \pi_{1,2} \circ \tF_{(\bx,\bxi, \bd)} \right) \pitchfork W\;  \text{with}\; j^{k}_{1}\left( \pi_{1,2} \circ \widetilde{F}_{(\bx,\bxi, \bd)} \right)(q_{0}) \in \overline{C_{j}}\rbrace
\end{align*}
is dense.

	Write $P(2,2,k) = \lbrace (P_{1}, P_{2}): P_{i}\; \text{is a polynomial with $P_{i}(0) = 0$ and deg$(P_{i}) \leq k$},\; i=1, 2 \rbrace$. Given $\bx \in C^{\infty}(U_{1,2}, \R^4)$ and $P =  (P_{1}, P_{2}) \in P(2,2,k)$, define $f_{(\bx, P)}: U_{0} \times I \times J \to \R^4$ by
\begin{align*}
	f_{(\bx, P)}(u,t) = \pi_{1,2} \circ \tF_{(\bx, \bxi)}(u,t) + P(u).
\end{align*}
Define also 
\begin{align*}
	\Phi: U_{1,2} \times I \times J \times P(2,2,k) &\to J^{k}(2,2)\\
	(u,t,P) & \mapsto j_{1}^{k}f_{(\bx,P)}(u,t)
\end{align*}
which is a submersion, thus $\Phi \pitchfork W$. Then, via lemma \ref{basictransvlemma} $$\lbrace P \in P(2,2,k): \Phi_{P} \pitchfork W\; \text{at}\; q_{0},\; \text{such that}\; \Phi_{P}(q_{0}) \in \overline{C_{j}}  \rbrace$$ is dense in $P(2,2,k)$. Hence, there is $\lbrace P_{n} \rbrace$ a sequence in $P(2,2,k)$ such that $P_{n} \to 0$, with $\Phi_{P_{n}} \pitchfork W$, for all $n \in \NN$. Note that $\bx_{n} = \bx + P_{n} \in T_{W,q_{0}, C_{j}, U_{1,2}}$, for all $n \in \NN$ and $\bx_{n} \to \bx$, therefore, $T_{W,q_{0}, C_{j}, U_{1,2}}$ is dense.
\end{proof}

If \begin{align}\label{deteq2}
	\det \begin{bmatrix}
		\bxi_{i} & \bd_{i}\\
		\bxi_{j} & \bd_{j}\\
	\end{bmatrix} \neq 0\, \text{at}\, u_{0},
\end{align}		
for $i,j \in \lbrace 1,2,3,4 \rbrace,\, i<j$ we define 		
\begin{align*}
	T^{(\bxi, \bd)}_{i,j,W, q_{0}} = \left\lbrace \bx \in C^{\infty}(U, \R^4): j^{k}_{1} \left( \pi_{i,j} \circ \widetilde{F}_{(\bx,\bxi, \bd)} \right) \pitchfork W \ \ at \ \ (u_{0},t_0,l_0)  \right\rbrace,
\end{align*}
where $\pi_{i,j}$ is the projection in the coordinates different than $i$ and $j$. Therefore, lemma \ref{Lema4.1} holds for these sets.

\begin{obs}\normalfont\label{obsconj}
Let us consider	the set ${\cal{R}}$ as the complement of
	\begin{align*}
		&\left\{ \mathbf{\Theta} \in {\cal{L}}: \rank \begin{pmatrix}
			\mathbf{\Theta} & \mathbf{\Theta}_{u_{1}} & \mathbf{\Theta}_{u_{2}}
		\end{pmatrix} < 3   \right\},
	\end{align*}
	where ${\cal{L}}$ is given in (\ref{setli}). Then it follows from lemma (\ref{Lema4.1}) that the set
	\begin{align*}
		T_{i,j,W, q_{0}} = \left\{ (\bx, \bxi, \bd): j^{k}_{1} \left( \pi_{i,j} \circ \widetilde{F}_{(\bx,\bxi, \bd)} \right) \pitchfork W \ \ at \ \ (u_{0},t_0,l_0),\; (\bxi, \bd) \in {\cal{R}}   \right\},
	\end{align*}
	is residual, for all $i,j \in \lbrace 1,2,3,4 \rbrace$, $i<j$.
\end{obs}		

\begin{teo}\label{teo4.1}
	There is an open dense set ${\cal{O}} \subset C^{\infty} \left( U, \mathbb{R}^4 \times (\R^{4} \setminus \lbrace  \bm{0}) \times (\R^{4} \setminus \lbrace  \bm{0}) \rbrace) \right)$, such that for all $(\bx, \bxi, \bd) \in \cal{O}$, the germ of the plane congruence $F_{(\bx, \bxi, \bd)}$ at any point $(u_{0}, t_{0}, l_{0}) \in U \times I \times J$ is stable. %Then, for all $(\bx, \bxi, \bd) \in \cal{O}$, the germ of the plane congruence $F_{(\bx, \bxi, \bd)}$ is an immersive germ or $\cA$-equivalent to one of the normal forms below:
	%	\begin{enumerate}[(a)]
	%	\item $(x,y,z,w) \mapsto (x,y,w, z^2)$ {\rm{(Fold)}}
%		\item $(x,y,z,w) \mapsto (x,y,w, z^3 + xz)$ {\rm{(Cusp)}}
%		\item $(x,y,z,w) \mapsto (x,y,w,z^4 + xz + yz^2)$ {\rm{(Swallowtail)}}
%		\item $(x,y,z,w) \mapsto (x,y,w,z^5 + xz + yz^2 + wz^3)$ {\rm{(Butterfly)}}
%		\item $\left( x,y,z,w \right) \mapsto (z, x^2 + y^2 + zx + wy, xy, w) $ {\rm{(Hyperbolic Umbilic)}}
%		\item $\left( x,y,z,w \right) \mapsto (z, x^2 - y^2 + zx + wy, xy, w) $ {\rm{(Elliptic Umbilic)}}.		
%	\end{enumerate}
\end{teo}

	\begin{proof} 
		Given $f :(\R^2, 0) \to \R^2$ be a smooth map germ and $z = j^{k}f(0)$, define
		\begin{align*}
			{\cal{K}}^{k}(z) = \lbrace j^kg(0): g \isEquivTo{\K} f \rbrace.
		\end{align*}
		For a sufficiently large $k$, define $$\Pi_{k}(2,2) = \lbrace f \in J^{k}(2,2):\, cod_{e}({\cal{K}}, f) \geq 5  \rbrace.$$ Consider $$\Sigma^{i} = \lbrace \sigma \in J^{1}(2,2): \text{kernel rank}(\sigma) = i \rbrace \subset J^{1}(2,2),$$
		which is a submanifold of  codimension $i^2$.
		
	%	\begin{enumerate}
			 Let us consider  $f \in \Pi_{k}(2,2)$ such that $\text{kernel rank}(df(0)) = 1$, that is $f \in \Pi_{k}(2,2) \cap \Sigma^{1} $. Hence, we write $f(x,y) = (x, g(x,y))$, where $g(0,y)$ has a singularity of $A_{r}$ type, for some $5 \leq r \leq k-1$. In this case, if we look at the complement of $\Pi_{k}(2,2)$ in $\Sigma^{1}$, we have ${\cK}$-singularities of $A_{1}$, $A_{2}$, $A_{3}$ and $A_{4}$-type. Furthermore, it is known that $\Pi_{k}(2,2) \cap \Sigma^{1} $ is a semialgebraic set of codimension greater than or equal to $5$, so it has a stratification $\lbrace {\cal{S}}_{i}^{1} \rbrace_{i=1}^{m_{1}}$, with codim$({\cal{S}}_{i}^{1})\geq 5$.

				 Let $f$ be a germ from $\R^2$ to $\R^2$ such that $\text{kernel rank}(df(0)) = 2$.  One can assume $f(x,y) = (g_{1}(x,y), g_{2}(x,y)) $, where $(g_{1}(x,y), g_{2}(x,y))$ has $2$-jet in $H^{2}(2,2)$, therefore, $(g_{1}(x,y), g_{2}(x,y))$ has $2$-jet given by one of the normal forms below (See \cite{Gibsonlivro} or \cite{juanjo}):
			\begin{align*}
				(x^2 + y^2, xy);\; (x^2-y^2, xy);\; (x^2, xy);\; (x^2, 0);\; (x^2 \pm y^2, 0);\; (0,0).
			\end{align*}
			The normal forms
			\begin{itemize}
				\item $W_{1}: \left(x^2 + y^2, xy \right)$
				\item $W_{2}: \left(x^2 - y^2, xy \right)$
			\end{itemize}
			have $cod_{e}({\cal{K}}) = 4$. The other ${\cal{K}}$-orbits have $cod_{e}({\cal{K}})\geq5$.
			Hence,  $\Pi_{k}(2,2) \cap \Sigma^{2}$ is a semialgebraic set of codimension greater than or equal to $5$ and admits a stratification $\lbrace {\cal{S}}_{i}^{2} \rbrace_{i=1}^{m_{2}}$ with $codim({\cal{S}}_{i}^{2})\geq 5$.

%\end{enumerate}

Then, it follows that the set of ${\cal{K}}$-orbits of codimension less than or equal to $4$, contains the following ${\cal{K}}$-orbits
\begin{itemize}
	\item type $A_{r}$, $1 \leq r \leq 4$;
	\item type $W_{1}$;
	\item type $W_{2}$.
\end{itemize}

Applying lemma (\ref{Lema4.1}) and remark (\ref{obsconj}) to each strata of the above stratification, we obtain that
\begin{align*}
\cT_{j} = \bigcap\limits_{i=1}^{m_{j}}{T}_{1,2,S_{i}^{j}, q_0},\; j=1,2\\
\cT_{2+r} = {T}_{1,2,A_{r}, q_0}, \, 1\leq r \leq 4 \\
\cT_{6+i} = {T}_{1,2,W_{i}, q_0}, \, i=1,2.
\end{align*} are residual subsets of $C^{\infty}(U, \R^{4} \times (\R^{4} \setminus \lbrace 0 \rbrace))$. Hence,
\begin{align*}
{\cal{O}}_{1,2,q_0} =  \bigcap\limits_{i=1}^{8}\cT_{i} 
\end{align*}
is residual. Similarly, we have that the sets ${\cal{O}}_{l,m,q_0}$, $l,m \in \lbrace 1,2,3,4 \rbrace$, $l<m$, are residual.

Since $\bxi$ and $\bd$ are linearly independent, at least one of the $2\times 2$ minors of
\begin{align*}
	M_{\bxi, \bd} = \begin{bmatrix}
		\bxi_{1} & \bd_{1}\\
		\bxi_{2} & \bd_{2}\\
		\bxi_{3} & \bd_{3}\\
		\bxi_{4} & \bd_{4}
	\end{bmatrix}
\end{align*}	
is different than zero at $q_0 = (u_{0}, t_{0}, l_{0})$, thus there is a residual subset $\cO_{q_0}$ such that
\begin{align*}
	(\bx,\bxi, \bd) \in {\cal{O}}_{q_0} \Leftrightarrow j^{k}_{1}\left( \tilde{\pi}_{l,m} \circ \widetilde{F}_{(\bx,\bxi, \bd)} \right) \pitchfork {\cal{A}}_{r},\, W_{1}, \, W_{2}, \, {\cal{S}}^{j}_{i},\, j=1,2, \, r=1, \cdots, 4.
\end{align*}
Thus, it follows from proposition \ref{propunfolding} that the germ of $\widetilde{F}_{(\bx,\bxi, \bd)}$ at $q_0$, which is equivalent to the germ of $F_{(\bx,\bxi, \bd)}$ at $q_0$, is a $2$-dimensional unfolding of $\widetilde{f}(u) = \pi_{l,m} \circ \widetilde{F}(u,t_0,l_0)$. From lemma (\ref{lema3.2}) we have that $F_{(\bx,\bxi, \bd)}$ is ${\cal{A}}$-infinitesimally stable for all $(\bx,\bxi, \bd) \in {\cal{O}}_{q_0}$ and therefore is ${\cal{A}}$-stable (see \cite{Mather}). Thus, there is a neighborhood $U_{u_{0}} \times I_{t_{0}} \times J_{l_0}$ of $q_{0} = (u_{0}, t_{0}, l_{0})$ in $U \times I \times J$, such that $F_{(\bx,\bxi, \bd)}{\big|}_{U_{u_{0}} \times I_{t_{0}}\times J_{l_0}}$ is ${\cal{A}}$-stable. This result does not dependent of the fixed point $q_{0}$, so we take a countable family of points $q_{i} = (u_{i}, t_{i}, l_{i}) \in U \times I \times J$ and neighborhoods $U_{u_{i}} \times I_{t_{i}} \times J_{l_{i}}$, $(i=1,2, \cdots)$, such that $F_{(\bx,\bxi, \bd)}{\big|}_{U_{u_{i}} \times I_{t_{i}}\times J_{l_{i}}}$ is ${\cal{A}}$-stable and 

\begin{align*}
	U \times I \times J = \bigcup\limits_{i=1}^{\infty} U_{u_{i}} \times I_{t_{i}} \times J_{l_{i}}.
\end{align*}
Since ${\cal{O}}_{q_{i}}$ is a residual subset of $ C^{\infty}(U, \R^{4} \times (\R^{4} \setminus \lbrace 0 \rbrace))$, it follows that
\begin{align*}
	{\cal{\tilde{O}}} = \bigcap\limits_{i=1}^{\infty} {\cal{O}}_{q_{i}}
\end{align*}
is residual. Furthermore, the germ of $F_{(\bx,\bxi, \bd)}$ at any point $(u,t,l) \in U \times I$ is ${\cal{A}}$-infinitesimally stable, for all $(\bx,\bxi, \bd) \in {\cal{\tilde{O}}} $.

Since ${\cal{F}}: C^{\infty}(U, \R^{4} \times (\R^{4} \setminus \lbrace 0 \rbrace) \times (\R^{4} \setminus \lbrace 0 \rbrace)) \rightarrow C^{\infty}(U \times I \times J, \R^{4})$, defined by ${\cal{F}}(\bx,\bxi, \bd)= F_{(\bx,\bxi, \bd)}$, is continuous and
\begin{align*}
	S=\lbrace f \in  C^{\infty}(U \times I\times J, \R^{4}): f \ \ {\cal{A}}-infinitesimally \ \ stable \rbrace
\end{align*}
is open (See \cite{gg} p. 111, for instance), ${\cal{O}} = {\cal{F}}^{-1}(S) $ is open. But ${\cal{\tilde{O}}}\subset {\cal{O}}$ and ${\cal{\tilde{O}}}$ is dense, therefore ${\cal{O}}$ is an open dense subset.

		\end{proof}

\subsection{Generic singularities of affine normal plane congruences}
Let $M = \bx(U)$ be a locally strictly convex surface in $\R^4$ and $\bxi$ a metric field. In this section, we consider the affine normal plane congruence associated to $M$, that is, the plane congruence $F_{(\bx, \mathbf{A})}$ given by $\bx: U \to \R^4$, $\bx(U) = M$ and its affine normal plane bundle $\mathbf{A}$, i.e.
\begin{align}
	F_{(\bx, \mathbf{A})}(u,t,l) = \bx(u) + t\bxi_{1}(u) + l\bxi_{2}(u),
\end{align}
where $\lbrace \bxi_{1}, \bxi_{2} \rbrace$ is a local frame of the affine normal plane $\bA$.

We seek to prove the following theorem, by using all the ingredients obtained from the study of the family of affine distance functions and its generic singularities in section \ref{sec3}.

\begin{teo}\label{teoprincipal}
	There is a residual subset $\cO \subset Emb_{lsc}(U, \R^4\times \R^4)$ such that the germ of the Blaschke exact normal plane congruence $F_{(\bx, \bA)}$ at any point $(u_{0}, t_{0}, l_{0}) \in U \times I \times J$ is a
	Lagrangian stable map germ for any $\bx \in \cO$, i.e., $ \forall\, \bx \in \cO$, $F_{(\bx,\bA)}$ is an immersive germ, or $\cA$-equivalent to one of the normal forms in table \ref{table1}.
	\begin{table}[!htbp]
		\setlength{\tabcolsep}{10pt}
		\renewcommand{\arraystretch}{1.2}
		\begin{center}
			\begin{tabular}{l l}
				\hline
				\textbf{Singularity} & \textbf{Normal form}  \\ \hline	
				Fold                                     &   $(x,y,w, z^2)$           \\	
				Cusp                                   &   $(x,y,w, z^3 + xz)$           \\ 	
				Swallowtail                               &  $(x,y,w,z^4 + xz + yz^2)$            \\ 	
				Butterfly                             &  $(x,y,w,z^5 + xz + yz^2 + wz^3)$            \\ 	
				Elliptic Umbilic                         &      $(z,w, x^2 - y^2 + zx, zy)$        \\ 	
				Hyperbolic Umbilic                     &        $(z,w, x^2 + zy, y^2 + zx, xy)$      \\ 	
				Parabolic Umbilic                       &    $(z,w, xy + xz, x^2 + y^3 + yw  )$ \\  \hline      
			\end{tabular}\caption{Generic singularities of affine plane congruences}\label{table1}
		\end{center}
	\end{table}
\end{teo}

Note that in theorem \ref{teoprincipal} we omit the reference to the metric field $\bxi$ in $F_{(\bx, \bA)}$, just to simplify notation. 

It follows from proposition \ref{propmorse} that the germ of family of affine distance functions $\Delta: \left(U \times \R^4, (u_0, p_0)  \right) \to (\R, s_0)$ is a Morse family of functions, thus, it has a Lagrangian immersion associated to it. Since the catastrophe set $C_{\Delta}$ is given by (see remark \ref{obscat})
\begin{align*}
		C_{\Delta} = \lbrace (u,p): p = \bx(u) - \lambda \bv,\; \text{where $\bv \in A_{\bx(u)}, \lambda \in \R$} \rbrace.
\end{align*}
it follows that the Lagrangian immersion associated to $\Delta$ is given by
\begin{align*}
	L(u,t,s) = \left( \bx(u) + t\bxi_{1}(u) + l\bxi_{2}(u), \dfrac{\partial \Delta}{\partial p} \right),
\end{align*}
where $\lbrace \bxi_{1}(u), \bxi_{2}(u) \rbrace$ generates $\bA_{\bx(u)}$. Hence, the Lagrangian map associated to the Morse family $\Delta$ is exactly the germ of affine normal plane congruence
\begin{align*}
	\pi \circ L (u,t,s) = \bx(u) + t\bxi_{1}(u) + l\bxi_{2}(u) =  F_{\left(\bx, \bA \right)}(u,t,l),
\end{align*}
where $\pi: T^{*}\R^4 \to \R^4$. Taking this into account, we prove theorem \ref{teoprincipal}.

\begin{proof}[Proof of theorem \ref{teoprincipal}]
	Let $F_{\left(\bx, \bA\right)}: \left(U \times I \times J, u_{0}, t_{0}, l_{0} \right) \to \left(
	\R^4, p_{0}\right)$ a germ of affine normal plane congruence. Note that $F_{\left(\bx, \bA\right)}$ is the Lagrangian map associated to $\Delta$. Since $\Delta$ is locally 
	$\cP$-$\mathcal{R}^{+}$-versal for a residual subset of $Emb_{lsc}(U, \R^4\times \R^4)$, thus the result follows from the fact that a Morse family is $\cP$-$\mathcal{R}^{+}$-versal if and only if its Lagrangian map is Lagrangian stable (see \cite{Livro}, theorem 5.4). 	
\end{proof}

Next, we characterize singularities of Blaschke exact normal plane congruences in terms of the eigenspaces associated to the eigenvalues of a shape operator of a vector in the affine normal plane.

\begin{teo}\label{caractsing}
	Let $F_{\left(\bx, \bA\right)}: U \times I \times J, \to \R^4$ be an exact Blaschke plane congruence and $q_0 = (u_0, t_0, l_0)$ a singular point of $F_{\left(\bx, \bA\right)}$. Then,
	\begin{enumerate}
		\item [(i)] $(t_0, l_0) \neq 0$.
		\item [(ii)] If $l_{0} \neq 0$ (resp. $t_0 \neq 0$), $q_0$ is a singular point of corank $i$ if and only if  $\dfrac{1}{l_0}$ (resp. $\dfrac{1}{t_0}$) is an eigenvalue of $S_{\bnu}$ with $\dim \text{Aut}(1/l_0) = i$ (resp. $\dim \text{Aut}(1/t_0) = i$), where $\bnu \in \bA_{\bx(u_0)}$ and $i=1,2$.
		\item [(iii)] $q_0 = (u_0, t_0, l_0)$ is a singular point of $F_{\left(\bx, \bA\right)}$ if and only if $\Delta_{p}$, where $p = \bx(u_0) + l_0 \bnu$, has a degenerate singularity at $u_0$.
	\end{enumerate}
\end{teo}

\begin{proof}
	\begin{enumerate}
		\item [(i)]	Let $F_{(\bx, \bA)} = F$ be the Blaschke plane congruence. Let us suppose that $q_0 = (u_0, t_0, l_0) \in \Sigma(F)$, then the Jacobian matrix of $F$ is the singular $4 \times 4$ matrix given by	
	\begin{align*}
	J(F)=	\begin{bmatrix}
			\bx_{u_{1}} + t_0 (\bxi_{1})_{u_1} + l_0(\bxi_{2})_{u_1} & 			\bx_{u_{2}} + t_0 (\bxi_{1})_{u_2} + l_0(\bxi_{2})_{u_2} & \bxi_{1} & \bxi_{2}
		\end{bmatrix}.
	\end{align*}	
	Since this matrix is singular, we have that $(t_0,l_0) \neq (0,0)$. 
	\item [(ii)] Let us suppose $l_0 \neq 0$. Note that
	\begin{align*}
		dF_{q_0}(\bv,a,b) = d\bx(u_0)(\bv) - t_0 d\bx(u_0)(S_{1}(\bv)) - l_0 d\bx(u_0)(S_{2}(\bv)) + a\bxi_{1} + b\bxi_{2} = 0
	\end{align*}                      
if and only if $a=b=0$ and $ \left(Id - l_0 \left( \dfrac{t_0}{l_0}S_1 + S_2\right) \right)(\bv) = 0$. If we denote $\bnu = \dfrac{t_0}{l_0}\bxi_{1} + \bxi_{2} \in \bA_{\bx(u_0)}$, thus by linearity $S_{\bv} =  \dfrac{t_0}{l_0}S_1 + S_2$ and $(\bv,0,0)$ belongs to the kernel of $dF_{q_0}$ if and only if 
\begin{align*}
\left(Id - l_0 S_{\bnu}\right)(\bv) = 0,
\end{align*}
that is, $\bv$ is an eigenvector of $S_{\bnu}$ with eigenvalue $\dfrac{1}{l_0}$. Moreover, $q_{0}= (u_0, t_0, l_0)$, with $l_0 \neq 0$ is a singular point of kernel rank $1$ (resp. kernel rank $2$) if and only if $\dim \text{Aut}(1/l_0) = 1$ (resp. $\dim \text(Aut)(1/l_0) = 2$).
\item [(iii)] It follows from proposition (ii) and lemma \ref{carahessiana}.
	\end{enumerate}
\end{proof}

{\footnotesize \bibliographystyle{siam}
    \bibliography{references}}

\end{document}